\documentclass[a4paper]{article}

\usepackage[english]{babel}
\usepackage[utf8]{inputenc}
\usepackage[margin=35mm]{geometry}
\usepackage[colorlinks,linkcolor=blue,anchorcolor=blue,citecolor=blue]{hyperref}
\usepackage{amsmath,amssymb}
\usepackage{amsthm}
\usepackage{graphicx}
\usepackage{mathrsfs}
\usepackage{xcolor}

\numberwithin{equation}{section}

\newtheorem{theorem}{Theorem}[section]
\newtheorem{lemma}[theorem]{Lemma}

\newtheorem{definition}[theorem]{Definition}
\newtheorem{corollary}[theorem]{Corollary}
\newtheorem{remark}[theorem]{Remark}

\newcommand{\R}{\mathbb{R}}

\newcommand{\N}{\mathbb{N}}

\newcommand{\B}{\mathfrak{B}}
\newcommand{\g}{\mathfrak{g}}

\begin{document}

\title{Existence of polyharmonic maps in critical dimensions}

\author{
Weiyong He
\footnote{Department of Mathematics, University of Oregon, Eugene, OR 97403, USA (whe@uoregon.edu)}
\qquad \and
Ruiqi Jiang
\footnote{Corresponding author. School of Mathematics, Hunan University, Changsha, 410082, P. R. China (jiangruiqi@hnu.edu.cn)}
\qquad \and
Longzhi  Lin
\footnote{Mathematics Department, University of California - Santa Cruz, 1156 High Street, Santa Cruz, CA 95064, USA (lzlin@ucsc.edu)}
}

\date{}
\maketitle

\begin{abstract}
We prove that for any two closed Riemannian manifolds $M^{2m}$ ($m\geq 1$) and $N$, there exists a minimizing (extrinsic) $m$-polyharmonic map for every free homotopy class in $[M^{2m}, N]$,  provided that
the homotopy group $\pi_{2m}(N)$ is trivial. This generalizes the celebrated existence results for harmonic maps and biharmonic maps. We also prove that there exists a non-constant smooth polyharmonic map from $\R^{2m}$ to $N$ by a blowup analysis at an energy-concentration point for an energy-minimizing sequence if the convergence fails to be strong.
\end{abstract}

\section{Introduction}
Let $(M,g)$, $(N,h)$ be smooth compact Riemannian manifolds without boundaries.  Assume that $(N, h)$ is isometrically embedded into an Euclidean space $\R^{K}$. Then we define
\begin{align*}
W^{k,p}(M,N)=\big \{u\in W^{k,p}(M,\R^K):\,\, u(x) \in N, \, a.e.\, x\in M \big \}
\end{align*}
where $k$ is a non-negative integer and $p\in [1,\infty)$.

In this paper, we consider the \emph{$m$-polyenergy} functional, for all $m \in \N^+$, $u\in W^{m,2}(M,N)$,
\begin{equation}\label{funct:main}
E_m (u)=\int_M \big|\Delta^{\frac{m}{2}} u \big|^2 dM_g:=
\left\{
\begin{aligned}
&\int_M |\nabla \Delta^{k-1} u|^2 \,dM_g, \quad m=2k-1, \\
&\int_M |\Delta^k u|^2 \,dM_g, \qquad m=2k,
\end{aligned}
\right.
\end{equation}
where $\nabla$ and $\Delta$ are Levi-Civita connection and Laplace-Beltrami operator on $(M, g)$ respectively, and $dM_g$ denotes the volume element of $(M, g)$. We call the critical points of $E_m(u)$ \emph{extrinsic $m$-polyharmonic maps}. The main result of this paper is the following:
\begin{theorem}\label{thm:main}
Suppose the dimension of $M$ is $2m$ and the homotopy group $\pi_{2m}(N)$ is trivial. For any free homotopy class $\alpha \in [M,N]$, there exists a smooth extrinsic $m$-polyharmonic map $u_\alpha:M \rightarrow N \subset \R^K$ which minimizes the energy functional $E_m(u)$ in the homotopy class $\alpha$.
\end{theorem}

When $m=1$, this is a classical result of Lemaire \cite{L79}, Sacks-Uhlenbeck \cite[Theorem 5.1]{SU0} and Schoen-Yau \cite{SY79}. When $m=2$, this is proved by C.Y. Wang \cite[Theorem B]{Wang} using the biharmonic map heat flow. Our method is motivated by \cite{He18} where the first author proved the existence of biharmonic almost complex structures in a fixed homotopy class. Similar ideas have been used by F. Lin  in \cite{Lin99} for $p$-harmonic maps. Our approach is different from the one  in the classical work of Sacks-Uhlenbeck \cite{SU0}, where the authors considered a perturbed elliptic system of the harmonic maps and proved energy concentration at isolated singularities. The elliptic system for polyharmonic maps is much more complicated and a perturbed elliptic system is even more complicated. We are not able to follow the approach as in \cite{SU0}.

Instead we consider an energy-minimizing sequence in a fixed homotopy class which converges weakly to a limit. By proving an $\epsilon$-regularity for the minimizing sequence, we conclude that if the convergence fails to be strong, then there has to be energy concentration at some isolated points. Then the condition that $\pi_{2m}(N)$ is trivial excludes energy concentration phenomenon and Theorem \ref{thm:main} follows.
We also show that there exists a non-constant smooth extrinsic polyharmonic map from $\R^{2m}$ to $N$ by blowup analysis provided that the energy concentration occurs.

Our $\epsilon$-regularity is different from the well-known regularity in the theory of harmonic maps by Schoen-Uhlenbeck \cite{SU2} since we do \textit{not} have an elliptic system to deal with.
To prove $\epsilon$-regularity for the minimizing sequence, the main technical point is to construct a new sequence of almost energy-minimizing sequence using extension theorem in $W^{m, 2}(M,N)$, as indicated in \cite{He18} in the case of $W^{2, 2}$ almost complex structures. Such extension theorem is only known previously for $W^{1, p}$ map in the theory of harmonic maps.
Establishing such extension theorem for $W^{m, 2}$ Sobolev maps in critical dimensions seems to be of independent interest.

Our method should have many applications in other setting, in particular for intrinsic polyharmonic maps (see the existence question raised by Eells and Lemaire in \cite[Problem (8.8)]{EL83}, even though there is a substantial technical difficulty in the intrinsic setting.
We shall consider this elsewhere.

The paper is organized as follows. In Section \ref{sec:prelim}, we gather some facts concerning the homotopy classes in critical Sobolev space $W^{k,p}(M,N)$ with $kp=dimM$. In Section \ref{sec:proof}, we focus on proving Theorem \ref{thm:main}. The $\epsilon$-regularity for an energy-minimizing sequence in a fixed homotopy class is established in Subsection \ref{subsec:regularity}. Under the condition that $\pi_{2m}(N)$ is trivial, the existence of minimizers is presented in Subsection \ref{subsec:existence}. In Section \ref{sec:blowup}, we prove that there exists at least one non-constant smooth extrinsic $m$-polyharmonic map from $\R^{2m}$ to $N$ if the energy-concentration points exist.

\medskip
\noindent \textbf{Convention:} Throughout the paper, we assume that $(M, g)$, $(N, h)$ are closed Riemannian manifolds and that $(N, h)$ is isometrically embedded into an Euclidean space $\R^{K}$.
\section{Preliminaries}\label{sec:prelim}
In this section, we recall some well-known results about the homotopy classes in critical Sobolev space $W^{k,p}(M,N)$ with $kp=n=dimM$, which was considered by White (see \cite{White1,White2}). For the convenience of reader, we show how to define homotopy classes in critical Sobolev space. First of all, we recall the density theorem for Sobolev spaces (see \cite{SU,Bethuel,BPV1,BPV2} and the references therein).

\begin{theorem}\label{thm:density}
If $kp\geq n=dim M$, then $C^\infty (M,N)$ is dense in $W^{k,p}(M,N)$ with respect to $W^{k,p}$-norm.
\end{theorem}

Then, let us introduce a very simple way to verify that two Lipschitz maps from $M$ to $N$ are homotopic.
\begin{theorem}[B. White \cite{White1}]\label{thm:lip_homotopic}
Suppose $f,g : M \rightarrow N$ are Lipschitz and $n=dim M$. Then there exists a positive constant $\eta_0$ such that if
\begin{align*}
\| f-g\|_{W^{1,n}}\leq \eta_0,
\end{align*}
then $f$ and $g$ are homotopic.
\end{theorem}

As a direct consequence of Theorem \ref{thm:lip_homotopic} and the Sobolev embedding theorem, we have the following result.
\begin{corollary}\label{coro:Wkp_homotopic}
Suppose $f,g\in C^\infty(M,N)$ and $kp=n=dimM$. Then there exists a positive constant $\eta$ depending only on $M,N,k,p$ such that if
$$\| f-g\|_{W^{k,p}}\leq \eta,$$
then $f$ and $g$ are homotopic.
\end{corollary}
It is worth pointing out that above result tells us that if two smooth maps are close enough to each other with respect to critical $W^{k,p}$-norm, then they are homotopic.

\medskip

We are now in a position to define the homotopy classes in critical Sobolev spaces.
\begin{definition}
Suppose $f,g \in W^{k,p}(M,N)$ and $kp=n=dimM$, we say $f$ is homotopic to $g$, also denoted by $f \sim g$, if for all $\delta>0$, there always exist two maps $f_{\infty}, g_{\infty}\in C^\infty (M,N)$ such that $f_{\infty}$ and $g_{\infty}$ are homotopic and
\begin{align*}
\|f-f_{\infty}\|_{W^{k,p}}+\|g-g_{\infty}\|_{W^{k,p}} \leq \delta.
\end{align*}
\end{definition}
It follows from Theorem \ref{thm:density} and Corollary \ref{coro:Wkp_homotopic} that the homotopic relation in critical Sobolev space is an equivalence relation and is a generalization of that in $C^0(M,N)$. Note that for any given homotopy class $\alpha$ in critical Sobolev spaces $W^{k,p}(M,N)$ (i.e., $kp=n=dimM$), there always exists $f_\alpha \in C^\infty(M,N) \in \alpha$. Moreover, there holds
\begin{align*}
\alpha=\{f\in W^{k,p}(M,N):\,f\sim f_\alpha \}
=\mbox{the $\|\cdot \|_{W^{k,p}}$ completion of}\, \{f \in C^\infty(M,N): f \sim f_\alpha \}.
\end{align*}
Then it follows that for $n=dim M=2m$ and $f_0 \in C^\infty(M,N)$, there holds that
\begin{align}\label{eqn:Em_equiv}
\inf \{E_m(u) \,|\, u \in W^{m,2}(M,N), \, u \sim f_0 \}
=\inf \{E_m(u) \,|\, u \in C^\infty (M,N), \, u \sim f_0 \}.
\end{align}

In practice, we can apply the same method as Theorem \ref{thm:lip_homotopic} to verify that two maps in critical Sobolev spaces are homotopic.
\begin{corollary}\label{coro:W1n_homotopic}
Suppose $f,g \in W^{k,p}(M,N)$ and $kp=n=dimM$. There exists a positive constant $\varepsilon_0$ such that if
\begin{align*}
\| f-g\|_{W^{1,n}}\leq \varepsilon_0,
\end{align*}
then $f$ and $g$ are homotopic.
\end{corollary}
\begin{proof}
It follows from Theorem \ref{thm:density} that for all $\delta >0$, there exists two maps $f_{\infty}, g_{\infty} \in C^\infty(M,N)$ such that
\begin{align*}
\|f-f_{\infty}\|_{W^{k,p}}+\|g-g_{\infty}\|_{W^{k,p}} \leq \delta.
\end{align*}
By Sobolev embedding theorem, we have
\begin{align*}
\|f_{\infty}-g_{\infty}\|_{W^{1,n}}
&\leq \|f-f_{\infty}\|_{W^{1,n}}+ \|g-g_{\infty}\|_{W^{1,n}}+ \|f-g\|_{W^{1,n}} \\
&\leq C \big( \|f-f_{\infty}\|_{W^{k,p}}+ \|g-g_{\infty}\|_{W^{k,p}} \big)
+ \|f-g\|_{W^{1,n}} \\
&\leq C \delta + \varepsilon_0.
\end{align*}
Then, by choosing suitable $\delta_0>0$ and $\varepsilon_0>0$ such that $C\delta_0+\varepsilon_0 \leq \eta_0$ where $\eta_0$ comes from Theorem \ref{thm:lip_homotopic}, we have that $f_\infty$ and $g_\infty$ are homotopic provided $\delta\leq \delta_0$ . Hence, $f$ and $g$ are homotopic by definition and the proof is completed.
\end{proof}

\section{Proof of Theorem \ref{thm:main}}\label{sec:proof}
Throughout the proof of Theorem \ref{thm:main}, we assume that the dimension of $M$ is $2m$. Hence, $W^{m,2}(M,N)$ is a critical Sobolev space.

Fix a free homotopy class $\alpha \in [M,N]$,  and denote
\begin{align*}
\mathcal{E}(m,\alpha)=\inf \{E_m(u) \,|\, u \in  \alpha \cap W^{m,2}(M,N)\}.
\end{align*}
Without loss of generality, we can set $\alpha=[f_\alpha]$ with $f_\alpha \in C^\infty(M,N)$ and $f_\alpha \in \alpha$. By (\ref{eqn:Em_equiv}), we know that
\begin{align}\label{eqn:mimi_value}
\mathcal{E}(m,\alpha)
& =\inf \{E_m(u) \,|\, u \in W^{m,2}(M,N), \, u \sim f_\alpha \} \notag \\
& =\inf \{E_m(u) \,|\, u \in C^\infty (M,N), \, u \sim f_\alpha \}.
\end{align}

Let $u_k \in C^\infty(M,N)$ be a minimizing sequence in the homotopy class $\alpha$, that is, $u_k \in \alpha$ and
\begin{align*}
\mathcal{E}(m,\alpha)= \lim_{k\rightarrow \infty} E_m(u_k).
\end{align*}
By Gagliardo-Nirenberg interpolation inequality and integration by parts, we have
\begin{align*}
\| u_k \|_{W^{m,2}} \leq C \big( E_m^{1/2}(u_k) + \| u_k\|_{L^\infty}\big) \leq C<\infty\,,
\end{align*}
where $C$ is a positive constant independent of $u_k$. Hence, up to a subsequence, we can assume that $u_k$ converges to $u^*$ weakly in $W^{m,2}$ and strongly in $W^{m-1,2}$, where $u^* \in W^{m,2}(M,N)$, that is,
\begin{align*}
u_k &\rightharpoonup u^*, \quad \mbox{in}\,\, W^{m,2}, \\
u_k &\rightarrow u^*, \quad \mbox{in}\,\, W^{m-1,2}.
\end{align*}

We will divide the proof of Theorem \ref{thm:main} into two parts. The first part is to establish the $\epsilon$-regularity for the minimizing sequence $\{u_k\}$. In this part our method is motivated by \cite{He18}.
The second part is to prove that $u^*$ is a smooth extrinsic $m$-polyharmonic map and that
 $u^*$ minimizes the functional $E_m(u)$ in homotopy class $\alpha$ under the condition $\pi_{2m}(N)=\{ 0\}$.

\subsection{\texorpdfstring{$\epsilon$}{epsilon}-regularity for the minimizing sequence \texorpdfstring{$\{u_k \}$}{uk}} \label{subsec:regularity}
Let us start by introducing the following measures which are totally bounded,
\begin{equation}\label{eqn:measure1}
\begin{array}{l}
\displaystyle \mu_k = \bigg(  \sum_{i=1}^m |\nabla^i u_k|^{\frac{2m}{i}}  \bigg) dM, \\
\displaystyle\mu^* = \bigg(  \sum_{i=1}^m |\nabla^i u^*|^{\frac{2m}{i}}  \bigg) dM, \\
\displaystyle\xi_k = \, \big|\Delta^{\frac{m}{2}} u_k \big|^2 dM, \\
\displaystyle\xi^* = \, \big|\Delta^{\frac{m}{2}} u^* \big|^2 dM.
\end{array}
\end{equation}
We may assume, by passing to a subsequence if necessary, $\mu_k$ weakly converges to a Radon measure $\mu$, and $\xi_k$ weakly converges to a Radon measure $\xi$. By Fatou's lemma, there exist Radon measures $\nu$ and $\eta$, called \emph{defect measures} (see for example \cite{Lin99}), such that
\begin{align*}
\mu=\nu+\mu^*,\quad \xi=\eta +\xi^*.
\end{align*}
Though defect measures $\nu$ and $\eta$ are different, they both measure the failure of strong convergence in $W^{m,2}$. Indeed, the following properties hold when $\nu$ or $\eta$ vanishes.
\begin{lemma}\label{lem:equi_defect}
The following three statements are equivalent:
\begin{enumerate}
\item  $u_k$ converges to $u^*$ strongly in $W^{m,2}(M,N)$.
\item $\nu \equiv 0$ in $M$.
\item $\eta \equiv 0$ in $M$.
\end{enumerate}
\end{lemma}
\begin{proof}
Since the proof is based on the following two simple observations, we omit further details. One is that $u_k$ weakly converges to $u^*$ in $W^{m,2}$ and strongly in $W^{m-1,2}$. The other is that for any $v \in W^{m,2}(M,N)$, by integration by parts, there holds
\begin{align*}
\int_M \big|\Delta^{\frac{m}{2}} v \big|^2 dM
= \int_M \big|\nabla^m v \big|^2 dM + \mbox{LOT},
\end{align*}
where LOT is a linear combination of the following terms
\begin{align*}
\int_M \nabla^i v \ast \nabla^j  v \,dM,
\quad \mbox{with} \,\, 1\leq i, j\leq m-1 \text{ and } i+j \leq 2m-2.
\end{align*}
Here $\ast$ denotes the contraction by the metric $g$ on $M$. For example, when $m=2$, there holds
\begin{equation}\label{contraction}
\int_M \big|\Delta v \big|^2 dM
= \int_M \big|\nabla^2 v \big|^2 dM + \int_M Ric(\nabla v, \nabla v)\, dM,
\end{equation}
where $Ric (\cdot,\cdot)$ denotes the Ricci curvature tensor on $(M,g)$, and we denote the second term in the right-hand side of \eqref{contraction} briefly by the contraction
\begin{align*}
\int_M \nabla v \ast \nabla v \, dM. \tag*{\qedhere}
\end{align*}
\end{proof}

\begin{remark}\label{rem:measures}
If we replace $M$ by some open subset $\Omega \subset M$ in above lemma, we have similar results with little modifications. That is, if $\eta \equiv 0$ in $\Omega$, then for any open subset $U \subset \subset \Omega$, there holds $\nu \equiv 0$ in $U$.
\end{remark}

\medskip
Before we proceed to prove the $\epsilon$-regularity for the minimizing sequence $\{ u_k\}$, let us recall the scaling invariance of the functionals we concern, such as
\begin{align*}
E_m(v,g)=\int_M | \Delta^{\frac{m}{2}} v |^2 dM_g
\quad \mbox{and}\quad
F(v,g)=\int_M \bigg(  \sum_{i=1}^m |\nabla^i v|^{\frac{2m}{i}}  \bigg) dM_g
\end{align*}
where $v \in W^{m,2}(M,N)$ with $dim M=2m$. That is, if $\lambda$ is a positive number and we do the scaling $g_{\lambda}:=\lambda^2 g$, then we have
\begin{align}\label{eqn:scal_inv}
E_m(v,g)=E_m(v,g_\lambda)
\quad \mbox{and} \quad
F(v,g)=F(v,g_\lambda).
\end{align}
Let $\tau_0$ be the injectivity radius of $(M,g)$. Then the injectivity radius of $(M,g_\lambda)$ is $\lambda \tau_0$. Hence, without loss of generality we can assume that, by choosing $\lambda$ large enough and replacing $(M,g)$ by $(M,g_\lambda)$ if necessary, the injectivity radius of $(M,g)$ is bigger than 2, i.e., $\tau_0>2$. For any $p\in M$ and $r \in (0,\tau_0)$, $B_r(p)$ denotes the geodesic ball of radius $r$ with center $p$. And we always choose the geodesic normal coordinates $(B_r,\{ x^i\})$ on geodesic ball $B_{r}(p)$, where $B_r$ denotes the ball of radius $r$ with center origin in Euclidean space $\R^{2m}$\,.

\begin{lemma}\label{lem:eps_regularity}
There exists a positive constant $\epsilon_0$ such that if
\begin{align*}
\mu (B_2(p))\leq \epsilon_0,
\end{align*}
then $\nu \equiv 0$ in $B_1(p)$.
\end{lemma}
\begin{proof}
Since the proof involves very technical and complicated constructions of an "almost energy-minimizing" sequence in the same homotopy class in $W^{m,2}(M,N)$, we will divide the proof into three steps.

We first set up notations which will be used throughout the proof. $\nabla$ denotes the Levi-Civita connection, $D$ denotes the ordinary derivatives in geodesic normal coordinates $(B_2, \{ x^i\})$, and $dM, dx$ stand for the volume element of $(M,g)$ and $\R^{2m}$ respectively. For simplicity of notation, the local coordinate $x=(x^1,\cdots, x^{2m})$ also stands for the point in $M$, if we write $x \in M$ or $x\in B_r (p)$. $C$ always denotes some positive constant depending only on $M$ and $N$.

\medskip
\noindent \textbf{Step One}: For any  fixed sufficiently large integer $j\in \N^+$, there exists a positive integer $\mathcal{L}_j$ such that for any $k \geq \mathcal{L}_j$, we can construct $\widetilde{u}_{k} \in W^{m,2}(M,N)$ such that
\begin{align}\label{eqn:almost_u_k}
\widetilde{u}_{k}(x)=
\begin{cases}
u^*(x), & x \in B_{1-j^{-1}}(p), \\
u_{k}(x), & x \in M \setminus B_1(p).
\end{cases}
\end{align}
Note that the following operations will be done in the local coordinates $(B_2,\{ x^i\})$. Firstly, let $\psi_j \in C^{\infty}(B_2, [0,1])$ be such that
\begin{align}\label{eqn:psi_j}
\psi_j(x) =
\begin{cases}
0, & x \in B_{1-j^{-1}}, \\
1, & x \in B_2 \setminus B_1,
\end{cases}
\quad \mbox{and} \quad
\|D^l \psi_j\|_{L^\infty(B_2)} \leq C j^{l}, \, \forall  \,j\geq 2\,, l\, \in \N.
\end{align}
Let us define
\begin{align}\label{ukj}
u_{k,j}(x)
= u_k(x) + \big(u^*(x)-u_k(x)\big)\big(1-\psi_j(x)\big), \quad x \in B_2.
\end{align}
which yields
\begin{align}
u_{k,j}(x)=
\begin{cases}
u^*(x), & x \in B_{1-j^{-1}}, \\
u_k(x), & x \in B_2 \setminus B_1.
\end{cases}
\end{align}

Secondly, let $\phi(x) \in C_0^\infty(B_1, [0,\infty))$ be a cut-off function on the unit ball $B_1 \subset \R^{2m}$ such that
\begin{align*}
\int_{B_1} \phi(x) dx =1\,.
\end{align*}
Set
\begin{align*}
\phi_\rho (x):= \rho^{-2m} \phi \bigg( \frac{x}{\rho} \bigg).
\end{align*}
Now we define
\begin{align}\label{ukjrho}
u_{k,j,\rho}:=u_{k,j,\rho(x)}(x)=\int_{B_1}\phi(z) u_{k,j}\big(x+ \rho(x) z\big) dz
\end{align}
where $\rho$ is only dependent of $|x|$ and we write $\rho(x)=\rho(|x|)$. Such technique was used in the theory of harmonic maps by Schoen and Uhlenbeck (\cite{SU2}). We can choose $\rho(x)=\rho(|x|) \in C^\infty(\R^{2m}, [0,\infty))$ (see Remark \ref{rem:rho} for an explicit choice of $\rho$) such that
\begin{equation}\label{cond:rho}
\left\{
\begin{aligned}
&\rho(s)>0, \quad \forall s \in (1-j^{-1},1), \\
&\rho(s)=0, \quad \forall s \in (-\infty,1-j^{-1}] \cup [1,\infty),\\
&\rho \big(1-\frac{1}{2}j^{-1} \big)=\overline{\rho}
:=\max \rho=c_1 j^{-m} \,\text{ for some universal constant } c_1>0 , \\
&\| \rho(s)\|_{C^{m}(\R)}\leq C.
\end{aligned}
\right.
\end{equation}
Then we have
\begin{align}\label{eqn:u_kjrho}
u_{k,j,\rho(x)}(x)=
\begin{cases}
u^*(x), & x \in B_{1-j^{-1}}, \\
u_k(x), & x \in B_2 \setminus B_1.
\end{cases}
\end{align}

Since $N$ is a closed Riemannian manifold isometrically embedded in Euclidean space $\R^K$, there exists a tubular neighbourhood $N_{\sigma_0}:=\{ p \in \R^K \,|\, \mbox{dist}(p, N) < \sigma_0 \}$ of $N$ such that the nearest point projection map $\pi: N_{\sigma_0} \rightarrow \R^K$ is well defined and smooth. We claim that for any fixed sufficiently large integer $j\in \N^+$ there exists a positive integer $\mathcal{L}_j$ such that for any $k \geq \mathcal{L}_j$, there holds
\begin{align*}
\mbox{dist}(u_{k,j,\rho(x)}(x), N) \leq \frac{1}{2}\sigma_0, \quad a.e. \, x \in B_2.
\end{align*}
The claim will be proved in {\bf Step Two}.
It follows from (\ref{eqn:u_kjrho}) that for any $k \geq \mathcal{L}_{j}$, there holds
\begin{align*}
\pi \big( u_{k,j,\rho(x)}(x)\big)=
\begin{cases}
u^*(x), & x \in B_{1-j^{-1}}, \\
u_k(x), & x \in B_2 \setminus B_1\,,
\end{cases}
\end{align*}
which implies that
\begin{align}\label{ukxequal}
\widetilde{u}_{k}(x) :=
\begin{cases}
\pi \big(u_{k,j,\rho(x)}(x)\big), & x \in B_2(p), \\
u_k(x), & x \in M \setminus B_2(p).
\end{cases}
\end{align}
is well defined and satisfies (\ref{eqn:almost_u_k}). It is a simple matter to check that $\widetilde{u}_k \in W^{m,2}(M,N)$.

\medskip
\noindent \textbf{Step Two}: In this step, we will prove the claim mentioned above. That is, for any fixed sufficiently large integer $j\in \N^+$, there exists a positive integer $L_j$ such that for any $k \geq L_j$, there holds
\begin{align}\label{ineqy:distN}
\mbox{dist}(u_{k,j,\rho(x)}(x), N) \leq \frac{1}{2}\sigma_0, \quad a.e.\, x \in B_2.
\end{align}

By (\ref{eqn:u_kjrho}), we know that $u_{k,j,\rho(x)}(x) \in N$, a.e. $|x| \in [0,1-j^{-1}] \cup [1,2)$. Hence, it suffices to show that \eqref{ineqy:distN} holds for a.e. $x$ such that $|x| \in (1-j^{-1}, 1)$. Fix $\mu_1>0$ sufficiently small which will be determined later. According to Remark \ref{rem:rho}, we can further assume that
\begin{equation}\label{cond:rho2}
\left\{
\begin{aligned}
&\rho(s) \leq c_2 \mu_1 \overline{\rho},
\quad \forall s \in  (1-j^{-1}, 1-j^{-1}+\mu_1 j^{-1}) \cup (1-\mu_1 j^{-1},1) \\
&\rho(s) \geq c_2 \mu_1 \overline{\rho},
\quad \forall s \in [1-j^{-1}+\mu_1 j^{-1}, 1-\mu_1 j^{-1}].
\end{aligned}
\right.
\end{equation}
where $c_2>0$ is a universal constant.

\medskip Denote
\begin{align*}
u^*_{\rho(x)}(x):= \int_{B_1} \phi(z) u^*(x+\rho(x)z) dz
= \int_{B_{\rho(x)}(x)} u^*(y) \phi_{\rho(x)}(y-x) dy
\end{align*}
where $\rho(x)>0$ as $|x| \in (1-j^{-1}, 1)$. Then by  a version of Poincar\'e inequality (see e.g. \cite[Section 8.11]{LL01}), for $|x| \in (1-j^{-1}, 1)$ we have
\begin{align}\label{ineqy:temp1}
\rho(x)^{-2m} \int_{B_{\rho(x)}(x)} \big| u^*(y)-u^*_{\rho(x)}(x) \big| dy
&\leq C \bigg( \int_{B_{\rho(x)}(x)} | D u^*(y)|^{2m} dy \bigg)^{\frac{1}{2m}} \notag \\
& \leq C \bigg( \int_{B_{\frac{3}{2}}} | D u^*(y)|^{2m} dy \bigg)^{\frac{1}{2m}} \notag \\
& \leq C \bigg( \int_{B_{\frac{3}{2}}(p)} | \nabla u^*(y)|^{2m} dM \bigg)^{\frac{1}{2m}} \notag \\
& \leq C \epsilon_0^{\frac{1}{2m}}\,,
\end{align}
where we used the fact $\mu(B_2) \leq \epsilon_0$ in the last inequality. It follows from (\ref{ineqy:temp1}) that for $|x| \in (1-j^{-1}, 1)$, there exists $y \in B_{\rho(x)}(x)$ (depending on $x$) such that
\begin{align}\label{ineqy:temp2}
u^*(y) \in N
\quad \mbox{and}\quad
|u^*(y)-u^*_{\rho(x)}(x)| \leq C \epsilon_0^{\frac{1}{2m}}.
\end{align}

\medskip
Let us focus on the interval $ |x| \in (1-j^{-1}, 1-j^{-1}+\mu_1 j^{-1})$ first. we compute
\begin{align*}
\big| u_{k,j,\rho(x)}(x)-u^*_{\rho(x)}(x) \big|
&= \bigg| \int_{B_1} \phi(z) \big( u_{k,j}-u^*\big)(x+\rho(x)z)  dz \bigg| \notag \\
&= \bigg| \int_{B_1} \phi(z) \big( (u_k-u^*)\psi_j \big)(x+\rho(x)z)  dz \bigg| \notag \\
& \leq C \max_{\substack{|x|\in (1-j^{-1},1-j^{-1}+\mu_1 j^{-1})\\ |z| \leq 1}}
\big\{\psi_j (x+\rho(x)z) \big\} \\
& \leq C \|D\psi_j\|_{L^\infty} \max \bigg| |x+\rho(x)z|- (1-j^{-1})\bigg|
\end{align*}
where in the first inequality we used the fact $\|u_k\|_{L^\infty(M)}+\|u^*\|_{L^\infty(M)}\leq C$ due to the fact that $u_k(x), u^*(x) \in N$ a.e., $x \in M$.
By \eqref{cond:rho2}, for $ |x| \in (1-j^{-1}, 1-j^{-1}+\mu_1 j^{-1})$, there holds
\begin{align*}
1-j^{-1}-2 \mu_1 \overline{\rho}
\leq |x+ \rho(x)z|
\leq 1-j^{-1}+\mu_1 j^{-1}+ 2\mu_1 \overline{\rho}.
\end{align*}
Thus, we obtain
\begin{align}\label{ineqy:temp3}
\big| u_{k,j,\rho(x)}(x)-u^*_{\rho(x)}(x) \big|
\leq C j (\mu_1 j^{-1}+ \mu_1 \overline{\rho})
\leq C \mu_1\,,
\end{align}
where we used (\ref{eqn:psi_j}) and (\ref{cond:rho}).
Combining (\ref{ineqy:temp2}) and (\ref{ineqy:temp3}), we have
\begin{align}\label{ineqy:distI}
\mbox{dist}(u_{k,j,\rho(x)}(x), N)
&\leq C \big| u_{k,j,\rho(x)}(x)-u^*(y) \big| \notag \\
&\leq C \big| u_{k,j,\rho(x)}(x)-u^*_{\rho(x)}(x) \big|+ C|u^*(y)-u^*_{\rho(x)}(x)| \notag \\
&\leq C \big( \epsilon_0^{\frac{1}{2m}} + \mu_1 \big)
\end{align}
for $|x| \in (1-j^{-1}, 1-j^{-1}+\mu_1 j^{-1})$.

\medskip
Now we consider the case $ |x| \in (1-\mu_1j^{-1}, 1)$.
Since $\mu_k \rightharpoonup \mu$ as Radon measures and $\mu(B_2) \leq \epsilon_0$, there exists a positive integer $\mathcal{L}_1$ such that for any $k \geq \mathcal{L}_1$ there holds
\begin{align}\label{ineqy:mu_k}
\mu_k \big(\overline{B}_{\frac{3}{2}} \big)
\leq \mu\big(\overline{B}_{\frac{3}{2}} \big) + \epsilon_0
\leq 2 \epsilon_0\,.
\end{align}
By a similar argument as above, replacing $u^*_{\rho(x)}(x)$ by
\begin{align*}
u_{k,\rho(x)}(x):=\int_{B_1} \phi(z) u_k(x+\rho(x)z) dz\,,
\end{align*}
we have that for any $k \geq \mathcal{L}_1$, (\ref{ineqy:distI}) holds for $|x| \in (1-\mu_1j^{-1}, 1)$\,.

\medskip
Finally, we deal with the case $ |x| \in [1-j^{-1}+\mu_1j^{-1},1-\mu_1j^{-1}]$. We compute
\begin{align*}
\big| u_{k,j,\rho(x)}(x)-u^*_{\rho(x)}(x) \big|
&= \bigg| \int_{B_1} \phi(z) \big( u_{k,j}-u^*\big)(x+\rho(x)z)  dz \bigg| \notag \\
&= \bigg| \int_{B_1} \phi(z) \big( (u_k-u^*)\psi_j \big)(x+\rho(x)z)  dz \bigg| \notag \\
& \leq C \rho(x)^{-2m} \int_{B_{\rho(x)}(x)} \big| u_k(y)-u^*(y) \big| dy \\
& \leq C \rho(x)^{-m}
\bigg( \int_{B_{\rho(x)}(x)} \big| u_k(y)-u^*(y) \big|^2 dy \bigg)^{\frac{1}{2}} \\
& \leq C (\mu_1 \overline{\rho})^{-m}
\bigg( \int_{B_{\frac{3}{2}}} \big| u_k(y)-u^*(y) \big|^2 dy \bigg)^{\frac{1}{2}} \\
& \leq C (\mu_1 \overline{\rho})^{-m}
\bigg( \int_{B_{\frac{3}{2}}(p)} \big| u_k(y)-u^*(y) \big|^2 dM \bigg)^{\frac{1}{2}}.
\end{align*}
Since $u_k$ converges to $u^*$ strongly in $W^{m-1,2}$, there exists a positive integer $\mathcal{L}_j>\mathcal{L}_1$ such that for any $k \geq \mathcal{L}_j$ we have
\begin{align}
\mu_1^{-m} j^{m^2} \|u_k-u^* \|_{L^2(M)}\leq j^{-1},
\end{align}
which implies that
\begin{align}\label{ineqy:temp4}
\big| u_{k,j,\rho(x)}(x)-u^*_{\rho(x)}(x) \big| \leq C j^{-1}.
\end{align}
Combining (\ref{ineqy:temp2}) and (\ref{ineqy:temp4}), we have that there exists a positive integer $\mathcal{L}_j$ such that for any $k \geq \mathcal{L}_j$ there holds
\begin{align}\label{ineqy:distIII}
\mbox{dist}(u_{k,j,\rho(x)}(x), N)
&\leq C \big| u_{k,j,\rho(x)}(x)-u^*(y) \big| \notag \\
&\leq C \big| u_{k,j,\rho(x)}(x)-u^*_{\rho(x)}(x) \big|+ C|u^*(y)-u^*_{\rho(x)}(x)| \notag \\
&\leq C \big( \epsilon_0^{\frac{1}{2m}} + j^{-1} \big)
\end{align}
for $ |x| \in [1-j^{-1}+\mu_1j^{-1},1-\mu_1j^{-1}]$.

In summary, by choosing $\mu_1=\epsilon_0^{\frac{1}{2m}}$ small enough such that $C\epsilon_0^{\frac{1}{2m}} \leq \frac{1}{12}\sigma_0$ , there exists a positive integer $\mathcal{L}_j$ (here without loss of generality, we can assume $j$ is large enough such that $C j^{-1}\leq \frac{1}{12}\sigma_0$)  such that (\ref{ineqy:distN}) holds for any $k \geq \mathcal{L}_j$. The claim is proved. Note that $\mu_1$ and $\epsilon_0$ are so chosen only dependent of the closed Riemannian manifolds $M,N$.

\medskip
\noindent \textbf{Step Three}: we will show that for any fixed sufficiently large integer $j\in \N^+$, there exists a positive integer $\mathcal{L}_j$ such that for any $k \geq \mathcal{L}_j$, the $\widetilde{u}_k$ obtained in {\bf Step One} is homotopic to $u_k$ in $W^{m,2}(M,N)$.
By Corollary \ref{coro:W1n_homotopic}, it is sufficient to show that
\begin{align}\label{ineqy:S3_temp0}
\|\widetilde{u}_k -u_k \|_{W^{1,2m}(M)} \leq \varepsilon_0\,,
\end{align}
where $\varepsilon_0$ comes from Corollary \ref{coro:W1n_homotopic}. By the construction of $\widetilde{u}_k$, we know that
\begin{align*}
\|\widetilde{u}_k -u_k \|_{W^{1,2m}(M)}
=\|\widetilde{u}_k -u_k \|_{W^{1,2m}(B_1(p))}
\leq C \|\widetilde{u}_k -u_k \|_{W^{1,2m}(B_1)}\,,
\end{align*}
where $W^{1,2m}(B_1)$ stands for the Sobolev space defined on unit ball $B_1$ in Euclidean space $\R^{2m}$.
We first compute
\begin{align}\label{ineqy:S3_temp1}
\|\widetilde{u}_k -u_k \|_{L^{2m}(B_1)}
& = \|\widetilde{u}_k -u_k \|_{L^{2m}(B_{1-j^{-1}})}
+ \|\widetilde{u}_k -u_k \|_{L^{2m}(B_1 \setminus B_{1-j^{-1}})} \notag \\
& \leq \|u^* -u_k \|_{L^{2m}(B_{1-j^{-1}})}
+ C\cdot \mbox{Vol} (B_1 \setminus B_{1-j^{-1}}) \notag \\
& \leq \|u^* -u_k \|_{L^{2m}(B_{1})}
+ C\cdot \mbox{Vol} (B_1 \setminus B_{1-j^{-1}}).
\end{align}
Since by \eqref{ukxequal} we have
\begin{align*}
\|D \widetilde{u}_k \|_{L^{2m}(B_1)}
&\leq C \|D u_{k,j,\rho(x)}(x) \|_{L^{2m}(B_1)} \notag \\
& \leq C \bigg\|\int_{B_1} \phi(z) D_x \big( u_{k,j}(x+\rho(x)z) \big) dz \bigg \|_{L^{2m}(B_1)} \notag \\
& \leq C \bigg\|\int_{B_1} \big|Du_{k,j}(x+\rho(x)z)\big| (1+ |\rho'| ) dz \bigg \|_{L^{2m}(B_1)} \notag \\
& \leq C \| Du_{k,j}\|_{L^{2m}(B_{\frac{3}{2}})} \notag \\
& \leq C \bigg( \| Du_k\|_{L^{2m}(B_{\frac{3}{2}})}
+ \| Du^*\|_{L^{2m}(B_{\frac{3}{2}})}
+ j \| u_k - u^*\|_{L^{2m}(B_{\frac{3}{2}})} \bigg),
\end{align*}
we obtain
\begin{align} \label{ineqy:S3_temp2}
\|D \widetilde{u}_k -Du_k\|_{L^{2m}(B_1)}
& \leq C \bigg( \| Du_k\|_{L^{2m}(B_{\frac{3}{2}})}
+ \| Du^*\|_{L^{2m}(B_{\frac{3}{2}})}
+ j \| u_k - u^*\|_{L^{2m}(B_{\frac{3}{2}})} \bigg) \notag \\
& \leq C \bigg(\epsilon_0^{\frac{1}{2m}} + j \| u_k - u^*\|_{L^{2m}(B_{\frac{3}{2}})} \bigg)
\end{align}
where we used (\ref{ineqy:mu_k}) in the last inequality.
Combining (\ref{ineqy:S3_temp1}) and (\ref{ineqy:S3_temp2}), we have
\begin{align*}
\|\widetilde{u}_k -u_k \|_{W^{1,2m}(M)}
\leq C \bigg(\mbox{vol} (B_1 \setminus B_{1-j^{-1}})
+\epsilon_0^{\frac{1}{2m}}
+ j \| u_k - u^*\|_{L^{2m}(B_{\frac{3}{2}})} \bigg)\,.
\end{align*}
Since $u_k$ converges to $u^*$ strongly in $W^{m-1,2}$, we choose $\epsilon_0$ small enough, $j$ large enough and then $\mathcal{L}_j$ sufficiently large such that (\ref{ineqy:S3_temp0}) holds for $k \geq \mathcal{L}_j$.

\medskip
\noindent \textbf{Step Four}: we are now in a position to prove $\nu \equiv 0$ in $B_1(p)$.

For any sufficiently large $j \in \N^+$, we choose $k_j \geq \mathcal{L}_j$ which will be determined later. Without loss of generality, we can assume $k_j$ strictly increases as $j$ goes to infinity. According to the conclusion in {\bf Step Three}, we know that $\widetilde{u}_{k_j}$ are homotopic to $u_{k_j}$ in $W^{m,2}(M,N)$.

Since $u_k$ is a minimizing sequence in homotopy class $\alpha$, it follows from (\ref{eqn:mimi_value}) that for any $\epsilon >0$, there exists $j_0>0$ such that for $j>j_0$, there holds
\begin{align*}
E_m(u_{k_j})
\leq \mathcal{E}(m,\alpha) + \epsilon
\leq E_m(\widetilde{u}_{k_j}) + \epsilon.
\end{align*}
By the construction of $\widetilde{u}_{k_j}$, we have
\begin{align*}
\int_{B_1(p)} |\Delta^{\frac{m}{2}} u_{k_j}|^2 dM
&\leq \int_{B_1(p)} |\Delta^{\frac{m}{2}} \widetilde{u}_{k_j}|^2 dM
+ \epsilon \\
&\leq \int_{B_{1-j^{-1}}(p)} |\Delta^{\frac{m}{2}} u^*|^2 dM
+\int_{B_1(p) \setminus B_{1-j^{-1}}(p)} |\Delta^{\frac{m}{2}} \widetilde{u}_{k_j}|^2 dM
+ \epsilon\,.
\end{align*}
Taking $j\rightarrow \infty$, we obtain
\begin{align*}
&\eta \big(B_1(p)\big) + \int_{B_1(p)} |\Delta^{\frac{m}{2}} u^*|^2 dM \\
\leq & \int_{B_1(p)} |\Delta^{\frac{m}{2}} u^*|^2 dM
+ \varliminf_{j\rightarrow \infty} \int_{B_1(p) \setminus B_{1-j^{-1}}(p)} |\Delta^{\frac{m}{2}} \widetilde{u}_{k_j}|^2 dM
+ \epsilon.
\end{align*}
By the arbitrariness of $\epsilon$, we obtain
\begin{align}\label{ineqy:S4_temp0}
\eta \big( B_1(p) \big)
\leq \varliminf_{j\rightarrow \infty}
\int_{B_1(p) \setminus B_{1-j^{-1}}(p)} |\Delta^{\frac{m}{2}} \widetilde{u}_{k_j}|^2 dM.
\end{align}

We now focus on estimating the right-hand side of (\ref{ineqy:S4_temp0}). A simple computation gives
\begin{align*}
\big|\Delta^{\frac{m}{2}} \widetilde{u}_{k_j}(x) \big|
\leq C \sum_{\beta \in \Lambda}
\big|\nabla u_{k_j,j,\rho(x)}(x)\big|^{\beta_1}
\big|\nabla^2 u_{k_j,j,\rho(x)}(x)\big|^{\beta_2}
\cdots
\big|\nabla^m u_{k_j,j,\rho(x)}(x)\big|^{\beta_m},
\end{align*}
where $\Lambda=\{\beta=(\beta_1,\cdots, \beta_m)\,|\, \beta_s \in \N, \, \sum_{s=1}^m s\beta_s=m\}$.  By Young's inequality, we have
\begin{align*}
\big|\Delta^{\frac{m}{2}} \widetilde{u}_{k_j}(x) \big|^2
\leq C \sum_{s=1}^m \big|\nabla^s u_{k_j,j,\rho(x)}(x)\big|^{\frac{2m}{s}}.
\end{align*}
Since by \eqref{ukj}, \eqref{ukjrho} and \eqref{cond:rho} we have
\begin{align*}
\big|\nabla^s u_{k_j,j,\rho(x)}(x)\big|
\leq C \int_{B_1}\sum_{t=1}^s \big| \big(\nabla^t u_{k_j,j} \big)(x+ \rho(x)z)\big| dz
\end{align*}
and
\begin{align*}
\big|\nabla^t u_{k_j,j} \big|
\leq C \bigg(
\big|\nabla^t u_{k_j} \big|
+ \big|\nabla^t u^* \big|
+ \sum_{l=0}^{t-1} j^{t-l} \big|\nabla^l u_{k_j}-\nabla^l u^* \big|
\bigg),
\end{align*}
it follows that
\begin{align*}
\int_{B_1(p) \setminus B_{1-j^{-1}}(p)} |\Delta^{\frac{m}{2}} \widetilde{u}_{k_j}|^2 dM
&\leq C  \sum_{s=1}^m \sum_{t=1}^s
\int_{U_{\overline{\rho}}} \big|\nabla^t u_{k_j,j}(x) \big|^{\frac{2m}{s}} dM\\
& \leq C
\sum_{s=1}^m \sum_{t=1}^s
\int_{U_{\overline{\rho}}}
\big|\nabla^t u_{k_j}(x) \big|^{\frac{2m}{s}}
+ \big|\nabla^t u^*(x) \big|^{\frac{2m}{s}}dM \\
& \quad + C
\sum_{s=1}^m \sum_{t=1}^s \sum_{l=0}^{t-1}
\int_{U_{\overline{\rho}}}
j^{(t-l)\frac{2m}{s}} \big|\nabla^l u_{k_j} -\nabla^l u^* \big|^{\frac{2m}{s}}
dM \\
&\leq C\bigg(
\int_{U_{\overline{\rho}}}
\sum_{s=1}^m  \big|\nabla^s u_{k_j}(x) \big|^{\frac{2m}{s}} dM
+ \mathcal{T}_1+\mathcal{T}_2+\mathcal{T}_3
\bigg) \\
& \leq C\bigg(
\mu_{k_j} \big( U_{\overline{\rho}}\big)
+ \mathcal{T}_1+\mathcal{T}_2+\mathcal{T}_3
\bigg)\,,
\end{align*}
where $\overline{\rho}=\max \rho=c_1 j^{-m}$ (see (\ref{cond:rho})) and
\begin{align*}
U_{\overline{\rho}}
&=\big \{x \in M \,|\, \mbox{dist} \big(x, B_1(p) \setminus B_{1-j^{-1}}(p)\big) \leq \overline{\rho} \big \}, \\
\mathcal{T}_1
&= \sum_{s=1}^m \sum_{t=1}^{s-1}
\int_{U_{\overline{\rho}}}
\big|\nabla^t u_{k_j}(x) \big|^{\frac{2m}{s}}dM, \\
\mathcal{T}_2
&= \sum_{s=1}^m \sum_{t=1}^{s}
\int_{U_{\overline{\rho}}}
\big|\nabla^t u^*(x) \big|^{\frac{2m}{s}}dM,  \\
\mathcal{T}_3
&= \sum_{s=1}^m \sum_{t=1}^s \sum_{l=0}^{t-1}
\int_{U_{\overline{\rho}}}
j^{(t-l)\frac{2m}{s}} \big|\nabla^l u_{k_j} -\nabla^l u^* \big|^{\frac{2m}{s}}
dM.
\end{align*}
We will show that $\mathcal{T}_i$'s ($i=1,2,3$) vanish as $j \rightarrow \infty$. Let us deal with the terms in $\mathcal{T}_1$ first. For $1\leq t <s \leq m$, there holds
\begin{align*}
\int_{U_{\overline{\rho}}}
\big|\nabla^t u_{k_j}(x) \big|^{\frac{2m}{s}}dM
\leq
\bigg( \int_{U_{\overline{\rho}}}
\big|\nabla^t u_{k_j}(x) \big|^{\frac{2m}{t}}dM \bigg)^{\frac{t}{s}}
\mbox{Vol}\big(U_{\overline{\rho}} \big)^{1-\frac{t}{s}}
\leq C \mbox{Vol}\big(U_{\overline{\rho}} \big)^{1-\frac{t}{s}},
\end{align*}
where we used the fact $\{u_{k_j}\}$ are uniformly bounded in $W^{m,2}$ in the last inequality. It follows that $\mathcal{T}_1$ vanishes as $j\rightarrow \infty$ since $\overline{\rho} = c_1 j^{-m}$. For $\mathcal{T}_2$, it is clear that $\mathcal{T}_2$ vanishes  as $j\rightarrow \infty$ for  $\mbox{Vol}\big(U_{\overline{\rho}} \big) \rightarrow 0$ as $j\rightarrow \infty$. Now we compute the terms in $\mathcal{T}_3$. For $0\leq l<t\leq s\leq m$, there holds
\begin{align*}
\int_{U_{\overline{\rho}}} j^{(t-l)\frac{2m}{s}}
\big|\nabla^l u_{k_j} -\nabla^l u^* \big|^{\frac{2m}{s}} dM
& \leq j^{2m^2} \int_{B_{\frac{3}{2}}}
\big|\nabla^l u_{k_j} -\nabla^l u^* \big|^{\frac{2m}{s}} dM \\
&\leq C j^{2m^2}
\|u_{k_j}-u^* \|_{W^{m-1,2}(M)}^{\frac{2m}{s}}\,,
\end{align*}
where in the last inequality we used the Sobolev inequality. Hence, we can choose $k_j \geq \mathcal{}L_j$ sufficiently large such that
\begin{align*}
j^{2m^2}
\|u_{k_j}-u^* \|_{W^{m-1,2}(M)} \leq j^{-1}\,,
\end{align*}
which implies $\mathcal{T}_3$ vanishes as $j \rightarrow \infty$. Therefore, by the above arguments, we have
\begin{align*}
\eta \big( B_1(p) \big)
\leq C \varliminf_{j\rightarrow \infty}
\mu_{k_j} \big( U_{\overline{\rho}}\big).
\end{align*}
Fix any small $\kappa>0$ and set
\begin{align*}
\overline{B}_{1,\kappa}=\{x\in M: \mbox{dist} \big(x, \partial B_1(p)\big) \leq \kappa \}\,,
\end{align*}
where $\partial B_1(p)$ denotes the boundary of geodesic ball $B_1(p)$.
It is clear that $U_{\overline{\rho}} \subset \overline{B}_{1,\kappa}$ for $j$ sufficiently large, which implies
\begin{align*}
\varliminf_{j\rightarrow \infty}
\mu_{k_j} \big( U_{\overline{\rho}}\big)
\leq \varliminf_{j\rightarrow \infty}
\mu_{k_j} \big( \overline{B}_{1,\kappa}\big)
\leq \mu \big( \overline{B}_{1,\kappa}\big)
\end{align*}
where in the second inequality we used the fact $\mu_{k_j} \rightharpoonup \mu$ as Radon measures. Since $\overline{B}_{1,\kappa}$ is a Borel set for any $\kappa \in (0,\frac{1}{2})$, it follows that
\begin{align*}
\lim_{\kappa \rightarrow 0^+} \mu \big( \overline{B}_{1,\kappa} \big)
= \mu \big( \partial B_1(p) \big)
= \nu \big( \partial B_1(p) \big)\,,
\end{align*}
which implies
\begin{align}
\eta \big( B_1(p) \big)
\leq C \varliminf_{j\rightarrow \infty}
\mu_{k_j} \big( U_{\overline{\rho}}\big)
\leq  C \nu \big( \partial B_1(p) \big)\,.
\end{align}
Similar arguments apply to the case $B_r(p) \subset B_2(p)$ for any $r\in[\frac{3}{4}, \frac{5}{4}]$ and we obtain
\begin{align*}
\eta \big( B_r(p) \big) \leq  C \nu \big( \partial B_r(p) \big),
\quad r\in \big[\frac{3}{4}, \frac{5}{4} \big]\,.
\end{align*}
The fact that $\nu (M)<\infty$ implies that there are at most countably many $r \in [\frac{3}{4}, \frac{5}{4}]$ such that $\nu \big( \partial B_r(p) \big)>0$. Hence, we can choose $r_0 \in (1, \frac{5}{4})$ such that $\nu \big( \partial B_{r_0}(p) \big)=0$ which leads to
\begin{align}
\eta(B_{r_0}(p))=0.
\end{align}
It follows from Remark \ref{rem:measures} that
\begin{align*}
\nu(B_1(p))=0.
\end{align*}
which is the desired conclusion.
\end{proof}

\begin{remark}
Let us illustrate the ideas of the proof of Lemma \ref{lem:eps_regularity}.
The key to the proof is to construct $\widetilde{u}_k \in W^{m,2}(M,N)$ which equals to $u^*$ in a small geodesic ball and equals to $u_k$ out of another small geodesic ball. It is natural for us to make use of cut-off function to achieve the goal. Thus, we obtain $u_{k,j}$. However, $u_{k,j}$ may not take values in $N$ almost everywhere in the annulus $B_1(p) \setminus B_{1-j^{-1}}(p)$. To overcome this problem, we want to apply the nearest point project $\pi: N_{\sigma_0}\rightarrow N$ . Then the main difficulty arises. That is to ensure that $u_{k,j}$ is sufficiently close to $N$ almost everywhere in $M$.
The successful resolution of the difficulty is based on the following two key observations.
\begin{itemize}
\item One is that the small energy condition $\mu(B_2(p))\leq \epsilon_0$ ensure that the average $u^*_{\rho}$, $u_{k,\rho}$ of $u^*$, $u_k$ respectively ($k$  is sufficiently large) are close to $N$ due to the Poincar\'e inequality.
\item The other is that $u_{k,j}$ is close to $u_k$ near the boundary of $B_1(p)$ (i.e., $|x| \in (1-\mu_1 j^{-1},1)$) , and is also close to $u^*$ near the boundary of $B_{1-j^{-1}}(p)$ (i.e., $|x| \in (1-j^{-1},1-j^{-1}+\mu_1 j^{-1})$). Therefore, we choose a suitable smooth radially symmetric function $\rho(x)$ supported in $|x| \in [1-j^{-1},1]$ such that the average $u_{k,j,\rho}$ of $u_{k,j}$ is close to $u_{k,\rho}$ in $|x| \in (1-\mu_1 j^{-1},1)$, and is close to $u^*_{\rho}$ in $|x| \in (1-j^{-1},1-j^{-1}+\mu_1 j^{-1})$. Moreover, since $\rho(x)$ has uniformly positive lower bound in the annulus $|x| \in [1-j^{-1}+\mu_1 j^{-1}, 1-\mu_1 j^{-1}]$, $u_{k,j,\rho}$ uniformly converges to $u^*_{\rho}$ in the annulus due to $u_k \rightarrow u^*$ in $W^{m-1,2}$.
\end{itemize}
A similar result was obtained by F. H. Lin in 1999 for the harmonic maps \cite[Theorem 3.1]{Lin99}, where the construction of extension was specific for $W^{1, p}$ maps. Our method is an adaption of the method used in \cite{He18}.
Finally, we specify the choice of the parameters $j$, $\mathcal{L}_j$ and $k_j$. It is sufficient for reader to keep in mind that $j$ is sufficiently large to ensure the volume of $B_1(p) \setminus B_{1-j^{-1}}(p)$ is sufficiently small, while $\mathcal{L}_j$ and $k_j$ are large enough is to make sure $\|u_k -u^*\|_{W^{m-1,2}}$ small enough.
\end{remark}

\begin{remark}\label{rem:rho}
Here we show explicitly how to construct a smooth function $\rho(s)$ that satisfies \eqref{cond:rho} and \eqref{cond:rho2}. Firstly, we find a smooth even function $f: \R \rightarrow [0,\infty)$ such that
\begin{equation*}
\left\{
\begin{aligned}
& f(s)>0, \quad \forall s \in \big(-\frac{1}{2},\frac{1}{2} \big),  \\
& f(s)=0, \quad \forall s \in (-\infty, -\frac{1}{2}] \cup [\frac{1}{2},\infty), \\
& f(s) \text{ is monotonically increasing on } (-\infty, 0], \\
& f(s) \text{ is monotonically decreasing on } [0, \infty).
\end{aligned}
\right.
\end{equation*}
Hence, we know that $c_1:= f(0)=\max \{ f(s): s\in \R \}>0$. Let
\begin{align*}
c_2:=\frac{f(-\frac{1}{2}+ \mu_1)}{\mu_1 f(0)}\,,
\end{align*}
where $\mu_1>0$ is from {\bf Step Two} in the proof of Lemma \ref{lem:eps_regularity} and it depends only on the closed Riemannian manifolds $M,N$.
Then we have
\begin{equation*}
\left\{
\begin{aligned}
& f(s) \leq c_2 \mu_1 f(0), \quad \forall s
\in \big(-\frac{1}{2}, -\frac{1}{2}+\mu_1) \cup \big(\frac{1}{2}-\mu_1, \frac{1}{2}), \\
& f(s) \geq c_2 \mu_1 f(0), \quad \forall s
\in \big[ -\frac{1}{2}+\mu_1, \frac{1}{2}-\mu_1 \big].
\end{aligned}
\right.
\end{equation*}
Now by scaling and translating,
\begin{align*}
\rho(s):=j^{-m} f \big(js+\frac{1}{2}-j \big)
\end{align*}
is the desired smooth function satisfying \eqref{cond:rho} and \eqref{cond:rho2}.
\end{remark}

Set
\begin{align}
\mathcal{S}:= \{p \in M \,|\, \lim_{r \rightarrow 0}
\varliminf_{k \rightarrow \infty} \mu_k \big( B_r(p) \big) \geq \frac{1}{2}\epsilon_0\}
\end{align}
where $\epsilon_0$ is defined in Lemma \ref{lem:eps_regularity}. A direct consequence of Lemma \ref{lem:eps_regularity} is the following.
\begin{lemma}\label{lem:support_nu}
The support of $\nu$ equals $\mathcal{S}$, which contains at most finitely many points.
\end{lemma}
\begin{proof}
It is clear that $\mathcal{S}$ is contained in the support of $\nu$. In fact, for any $p \in \mathcal{S}$, $r \in (0,\tau_0)$ ($\tau_0$ is the injectivity radius of $M$), it follows from $\mu_k \rightharpoonup \mu$ that
\begin{align*}
\mu( \overline{B}_r(p))
\geq \varlimsup_{k \rightarrow \infty} \mu_k ( \overline{B}_r(p))
\geq \varliminf_{k \rightarrow \infty} \mu_k ( B_r(p))
\geq \frac{1}{2} \epsilon_0,
\end{align*}
where $\overline{B}_r(p)$ denotes the closure of open geodesic ball $B_r(p)$. Hence,
\begin{align}\label{ineqy:nu_temp0}
\nu(\{p\})= \lim_{r \rightarrow 0} \nu( \overline{B}_r(p))
=\lim_{r \rightarrow 0}
\bigg( \mu( \overline{B}_r(p)) -\mu^* (\overline{B}_r(p)) \bigg)
\geq \frac{1}{2} \epsilon_0,
\end{align}
which is the desired claim. Now we will show that the support of $\nu$ is also contained in $\mathcal{S}$. It suffices to prove that for any $p \in M \setminus \mathcal{S}$, there exists $r_0>0$ such that $\nu (B_{r_0}(p))$=0.
By the definition of $\mathcal{S}$, we know that there exists $r_1 >0$ such that
\begin{align*}
\varliminf_{k \rightarrow \infty} \mu_k \big( B_{r_1} (p) \big)
< \frac{1}{2}\epsilon_0,
\end{align*}
which implies
\begin{align*}
\mu( B_{r_1} (p))
\leq \varliminf_{k \rightarrow \infty} \mu_k \big( B_{r_1} (p) \big)
< \frac{1}{2}\epsilon_0,
\end{align*}
By Lemma \ref{lem:eps_regularity}, we know that $\nu (B_{r_0}(p))=0$ with $r_0=\frac{1}{2} r_1$. Finally, it follows from (\ref{ineqy:nu_temp0}) and $\nu(M)<\infty$ that $\mathcal{S}$ contains at most finitely many points. The proof is completed.
\end{proof}

\subsection{Existence of minimizers}\label{subsec:existence}
In this subsection, we first prove that $u^*$ is a smooth extrinsic $m$-polyharmonic from $M$ into $N$. Then we show that $\mathcal{S}$ is an empty set under the assumption $\pi_{2m}(N)=\{0\}$, which implies Theorem \ref{thm:main} due to Lemma \ref{lem:equi_defect}. We begin by proving the following lemma.
\begin{lemma}\label{lem:u*smooth}
Suppose $\Omega$ is an open subset of $M$. If $u_k$ converges to $u^*$ strongly in $W^{m,2}(\Omega, N)$, then $u^*$ is an extrinsic $m$-polyharmonic map from $\Omega$ to $N$. In particular, $u^* \in C^\infty(\Omega, N)$.
\end{lemma}
\begin{proof}
Due to Gastel and Scheven (\cite{GS}), we know that weakly $m$-polyharmonic maps (extrinsic or intrinsic) from $M$ to $N$ in critical dimensions  are smooth. Thus we only need to show that $u^*$ is a weakly $m$-polyharmonic map from $\Omega$ to $N$.

For any $\chi \in C_0^\infty(\Omega, \R^K)$ fixed, we define
\begin{align}
u_k^t(x) := \pi \bigg( u_k(x)+ t\, \chi(x) \bigg),
\quad t \in (-\delta_0 , \delta_0), \,x\in M\,,
\end{align}
where $\delta_0>0$ is sufficiently small and independent of $u_k$, and $\pi$ is the nearest point projection (see the proof of Lemma \ref{lem:eps_regularity}). It is obvious that $u_k^t \in C^{\infty}(M,N)$ is homotopic to $u_k$, thus $u_k^t \in \alpha$. We compute
\begin{align*}
E_m(u_k^t)= \int_M \big|\Delta^{\frac{m}{2}} u_k^t\big|^2 dM
= E_m(u_k) + 2 t \int_M \big \langle \Delta^{\frac{m}{2}}u_k, \Delta^{\frac{m}{2}} \big( D\pi(u_k) \chi \big) \big\rangle dM + O_k(t^2).
\end{align*}
Note that by estimating the second derivative of $E_m(u_k^t)$ with respect to $t$, we can obtain $|O_k(t^2)| \leq C t^2$ where $C$ is independent of $k$.
Since $u_k$ is the minimizing sequence in the homotopy class $\alpha$ and $u_k^t \in \alpha$, it follows that
\begin{align*}
\varliminf_{k \rightarrow \infty} E_m(u_k^t)
\geq \varliminf_{k \rightarrow \infty} E_m(u_k)\,,
\quad \forall t \in (-\delta_0, \delta_0),
\end{align*}
which implies
\begin{align}\label{eqn:weaksol_temp1}
\lim_{k \rightarrow \infty} \int_M \big \langle \Delta^{\frac{m}{2}}u_k, \Delta^{\frac{m}{2}} \big( D\pi(u_k) \chi \big) \big\rangle \, dM =0.
\end{align}

We claim that if $u_k$ converges strongly to $u^*$ in $W^{m,2}(\Omega, N)$, then we have
\begin{align}
\lim_{k \rightarrow \infty} \int_M \big \langle \Delta^{\frac{m}{2}}u_k, \Delta^{\frac{m}{2}} \big( D\pi(u_k) \chi \big) \big\rangle \,dM
=
\int_M \big \langle \Delta^{\frac{m}{2}}u^*, \Delta^{\frac{m}{2}} \big( D\pi(u^*) \chi \big) \big\rangle \,dM.
\end{align}
If this claim holds, it follows from (\ref{eqn:weaksol_temp1}) that
\begin{align*}
\int_M \big \langle \Delta^{\frac{m}{2}}u^*, \Delta^{\frac{m}{2}} \big( D\pi(u^*) \chi \big) \big\rangle \,dM =0.
\end{align*}
Hence, $u^*$ is a weakly $m$-polyharmonic map from $\Omega$ to $N$ and the desired result is obtained.

\medskip
Now we prove the claim above. For simplicity, we only give the proof of the claim for the case $m=2$; the other cases are similar in nature. That is,
\begin{align}\label{eqn:weaksol_temp2}
\lim_{k \rightarrow \infty}
\int_M \big \langle \Delta u_k, \Delta \big( D\pi(u_k) \chi \big) \big\rangle dM
=
\int_M \big \langle \Delta u^*, \Delta \big( D\pi(u^*) \chi \big) \big\rangle dM.
\end{align}
We compute
\begin{align}\label{eqn:weaksol_temp3}
\Delta \big( D\pi(u_k) \chi \big)
=& D^2 \pi (u_k) (\chi, \Delta u_k)
+ D^3\pi (u_k) (\chi, \nabla u_k) \nabla u_k  \notag \\
& +2 D^2 \pi (u_k) (\nabla \chi, \nabla u_k)
+D\pi (u_k) \Delta \chi.
\end{align}
We first deal with the term $D^2 \pi (u_k) (\chi, \Delta u_k)$ as an illustration. Since
\begin{align*}
& \bigg| \big \langle \Delta u_k, D^2 \pi (u_k) (\chi, \Delta u_k) \big) \big\rangle
- \langle \Delta u^*, D^2 \pi (u^*) (\chi, \Delta u^*) \big\rangle \bigg| \\
\leq&  \,C \bigg(
| \Delta u_k - \Delta u^*| \, |\Delta u_k|
+  |\Delta u^*| \,| \Delta u_k - \Delta u^*|
+ |\Delta u^*|^2 |D^2 \pi (u_k)-D^2 \pi (u^*)|
\bigg) |\chi| \\
\leq&  \,C \bigg(
| \Delta u_k - \Delta u^*| \, |\Delta u_k|
+ | \Delta u_k - \Delta u^*| |\Delta u^*|
+ |\Delta u^*|^2 |u_k-u^*|
\bigg)|\chi|\,,
\end{align*}
where $C$'s stand for positive constants independent of $u_k$, we have
\begin{align*}
& \int_M \bigg| \big \langle \Delta u_k, D^2 \pi (u_k) (\chi, \Delta u_k) \big) \big\rangle
- \langle \Delta u^*, D^2 \pi (u^*) (\chi, \Delta u^*) \big\rangle \bigg| dM \\
\leq & \,C \bigg(
\|\Delta u_k -\Delta u^*\|_{L^2(\Omega)}
\big( \|\Delta u_k \|_{L^2(\Omega)} + \|\Delta u^* \|_{L^2(\Omega)} \big)
+\int_{\Omega} |\Delta u^*|^2 |u_k-u^*| dM
\bigg).
\end{align*}
Since $\|u_k-u^*\|_{W^{2,2}(\Omega)} \rightarrow 0$ as $k \rightarrow \infty$, it is obvious that
\begin{align*}
\|\Delta u_k -\Delta u^*\|_{L^2(\Omega)}
\bigg( \|\Delta u_k \|_{L^2(\Omega)} + \|\Delta u^* \|_{L^2(\Omega)} \bigg) \rightarrow 0 \quad \mbox{as }k \rightarrow \infty.
\end{align*}
By dominated convergence theorem, we also obtain
\begin{align*}
\lim_{k \rightarrow \infty}\int_{\Omega} |\Delta u^*|^2 |u_k-u^*| dM =0.
\end{align*}
Hence,
\begin{align*}
\lim_{k \rightarrow \infty}
\int_{M}\big \langle \Delta u_k, D^2 \pi (u_k) (\chi, \Delta u_k) \big) \big\rangle
= \int_{M} \big \langle \Delta u^*, D^2 \pi (u^*) (\chi, \Delta u^*) \big\rangle dM\,.
\end{align*}
By similar arguments one can deal with the other terms in (\ref{eqn:weaksol_temp3}) and thus (\ref{eqn:weaksol_temp2}) is proved.{\qedhere}
\end{proof}

An immediate consequence of Lemma \ref{lem:u*smooth} is the following.
\begin{theorem}\label{thm:u*smooth}
$u^*$ is an extrinsic $m$-polyharmonic from $M$ into $N$. In particular, $u^* \in C^\infty(M,N)$.
\end{theorem}
\begin{proof}
By Lemma \ref{lem:support_nu} and Lemma \ref{lem:equi_defect}, we know that for any open set $\Omega \subset \subset M \setminus \mathcal{S}$, $u_k$ converges to $u^*$ strongly in $W^{m,2}(\Omega,N)$. It follows from Lemma \ref{lem:u*smooth} that $u^*\in C^\infty (M \setminus \mathcal{S}, N)$ is an extrinsic $m$-polyharmonic map from $M \setminus \mathcal{S}$ into $N$. Since $\mathcal{S}$ contains at most finitely many points and $u^*\in W^{m,2}(M,N)$, we conclude that $u^*$ is a weakly $m$-polyharmonic map from $M$ into $N$. Hence, the proof is completed due to the regularity result of Gastel and Scheven \cite{GS}.
\end{proof}

We now proceed to show that $\mathcal{S}$ is empty provided that $\pi_{2m}(N)$ is trivial.
\begin{theorem}\label{thm:S-empty}
If $\pi_{2m}(N) = \{0\}$, then $\mathcal{S}$ is empty.
\end{theorem}
\begin{proof}
Suppose otherwise that there exists a point $p\in \mathcal{S}$. By Lemma \ref{lem:support_nu}, $\mathcal{S}$ contains at most finitely many points, and so we can choose a geodesic ball $B_{3R_0}(p)$ such that $B_{3R_0}(p) \cap \mathcal{S}=\{ p\}$. This implies that the support of measure $\nu$ only contains one point $p$ on the domain $B_{3R_0}(p)$. Hence, for any $r_0 \in (0, 2R_0)$, there holds
\begin{align}\label{ineqy:S-empty:temp0}
\lim_{k\rightarrow \infty}
\|u_k -u^* \|_{W^{m,2}\big( B_{2R_0}(p) \setminus B_{r_0}(p)\big)} =0.
\end{align}
Without loss of generality, we can assume that, by scaling and choosing a subsequence if necessary, $R_0=1$ and for $k \geq \mathcal{K}$ ($\mathcal{K}$ is a sufficiently large positive integer) there holds
\begin{align}\label{ineqy:S-empty:temp1}
\| u_k \|_{W^{m,2}\big(B_2(p)\setminus B_{\frac{1}{4}}(p)\big)}
+ \| u^* \|_{W^{m,2}(B_2(p))}
\leq \theta< <\epsilon_0,
\end{align}
where $\epsilon_0$ is defined in Lemma \ref{lem:eps_regularity}, and $\theta$ is a sufficiently small positive number. We will show that \eqref{ineqy:S-empty:temp1} together with the condition $\pi_{2m}(N)=\{0\}$ enable us to push through similar arguments as in the proof of Lemma \ref{lem:eps_regularity}, which yields $\nu =0$ on $B_{\frac{1}{2}}(p)$ and thus $p$ is not an energy-concentration point (i.e., $p \notin \mathcal{S}$). This leads to a contradiction.
\begin{enumerate}
\item The small energy condition (\ref{ineqy:S-empty:temp1}) for $u_k$ and $u^*$ on the annulus $B_2(p) \setminus B_{\frac{1}{4}}(p)$ is sufficient to guarantee that {\bf Step One} and {\bf Step Two} in the proof of Lemma \ref{lem:eps_regularity} are valid. Hence, we obtain  an "almost energy-minimizing" sequence $\widetilde{u}_k$ for large $k$'s. Since $u^*\in C^{\infty}(M,N)$ (see Theorem \ref{thm:u*smooth}), we know $\widetilde{u}_k \in C^{\infty}(M,N)$ and therefore $\widetilde{u}_k \in W^{m,2}(M,N)$.

\item By the construction of $\widetilde{u}_k$, we have
    \begin{align}\label{ineqy:S-empty:temp3}
    \widetilde{u}_k |_{M \setminus B_1 (p)}= u_k |_{M \setminus B_1 (p)}\,,
    \end{align}
    where $B_1 (p)$ denotes the open geodesic ball. We know that the $2m$-dimensional sphere ${\mathbb{S}}^{2m}$ can be obtained by gluing two $2m$-dimensional balls (e.g., $B_1(p)$) along their boundaries. Thus we can define a continuous map $\Phi_k$ from ${\mathbb{S}}^{2m}$ into $N$ by gluing maps $\widetilde{u}_k$ and $u_k$ along the boundary of $B_1(p)$ due to \eqref{ineqy:S-empty:temp3}. The condition $\pi_{2m}(N)=\{ 0\}$ implies that $\Phi_k$ is homotopic to a constant map. Hence, we deduce that $\widetilde{u}_k$ is homotopic to $u_k$, i.e., $\widetilde{u}_k \in \alpha$, and {\bf Step Three} in the proof of Lemma \ref{lem:eps_regularity} holds.

\item Doing the same computation as in {\bf Step Four} in the proof of Lemma \ref{lem:eps_regularity} and using \eqref{ineqy:S-empty:temp0} we obtain
    \begin{align}
   \varliminf_{j\rightarrow \infty} \int_{B_1(p) \setminus B_{1-j^{-1}}(p)} |\Delta^{\frac{m}{2}} \widetilde{u}_{k_j}|^2 dM =0\,.
    \end{align}
Note that the condition \eqref{ineqy:S-empty:temp0} simplifies the computation in {\bf Step four}. Hence, $\eta =0$ on $B_1(p)$ which implies $\nu=0$ on $B_{\frac{1}{2}}(p)$. \qedhere
 \end{enumerate}
\end{proof}

\section{Blowup analysis at the concentration point}\label{sec:blowup}
In this subsection, we will show that there exists at least one non-constant $m$-polyharmonic map $v:\R^{2m} \rightarrow N$ provided that $\mathcal{S}$ is nonempty. Before going further, let us introduce some notations to emphasize the metric of the source manifold: for any closed Riemannian manifold $(M,g)$, we define
\begin{itemize}
\item $B_{r}(p;g)$ : the geodesic open ball of radius $r>0$ centered at $p$ on $(M,g)$;
\item $\B_{r}$ : the Euclidean open ball of radius $r>0$ centered at $0$ in $\R^{2m}$;
\item $\iota_g$ : the injectivity radius of $(M,g)$;
\item $\Phi_{p,g}: \B_{\iota_g} \rightarrow B_{\iota_g}(p;g) \subset (M,g)$ is a geodesic normal coordinate for geodesic ball $B_{\iota_g}(p;g)$ such that $\Phi_{p,g}(0)=p$, and if we write
    \begin{align}\label{nota:g_ij}
    \Phi_{p,g}^*g=g_{ij}(x;p)\, dx^i\otimes dx^j, \quad \forall x \in \B_{\iota_g},
    \end{align}
    where we used $g_{ij}(x;p)$ to emphasize the dependence on $p$ for the geodesic normal coordinates, we have $g_{ij}(0;p)=\delta_{ij}$.
\end{itemize}

If $\mathcal{S}$ is nonempty, let us fix $p^* \in \mathcal{S}$ and focus on the convergence of $u_k$ near the concentration point $p^*$.
Since $\mathcal{S}$ contains at most finitely many points, there exists a geodesic ball $B_{6\tau_0}(p^*;g)$ for some $\tau_0>0$ such that
\begin{align}\label{eqn:p^*}
B_{6\tau_0}(p^*;g) \cap \mathcal{S}=\{ p^*\}.
\end{align}
Now choose $r_k \in (0, \tau_0)$ and $q_k \in \overline{B}_{\tau_0}(p^*;g)$ such that
\begin{align}\label{eqn:rk_qk}
\mu_k \big( B_{3r_k}(q_k;g) \big)
= \sup_{q \in \overline{B}_{\tau_0}(p^*;g)} \mu_k \big( B_{3r_k}(q;g) \big)
=\frac{\epsilon_0}{8}\,,
\end{align}
where $\epsilon_0$ comes from Lemma \ref{lem:eps_regularity}. We claim that
\begin{align}\label{lim:rk_qk}
r_k \rightarrow 0,
\quad
q_k \rightarrow p^*.
\end{align}
In fact, if $r_k \rightarrow 0$ fails, then there exists $r_0>0$ and a subsequence of $\{r_k\}$, also denoted by $r_k$, such that $\varliminf \limits_{k \rightarrow \infty} r_k \geq r_0$. By (\ref{eqn:rk_qk}), we know
\begin{align*}
\mu_k \big(B_{r_0}(q;g) \big) \leq \frac{\epsilon_0}{8}, \quad \forall q \in B_{\tau_0}(p^*;g),
\end{align*}
which implies
\begin{align*}
\mu \big(B_{r_0}(q;g) \big) \leq \frac{\epsilon_0}{8}, \quad \forall q \in B_{\tau_0}(p^*;g).
\end{align*}
It follows from Lemma \ref{lem:eps_regularity} that $u_k$ converges strongly to $u^*$ on $B_{\tau_0}(p^*;g)$, which contradicts the fact that $p^* \in \mathcal{S}$. If $q_k \rightarrow p^*$ fails, the limit of $q_k$ is another concentration point in $\overline{B}_{\tau_0}(p^*;g)$ (see Lemma \ref{lem:support_nu}), which contradicts (\ref{eqn:p^*}).

\begin{remark}
It is well known that the sequence for blowup analysis usually satisfies some kind of equations, such as the minimizers of perturbed functionals used by Sacks and Uhlenbeck \cite{SU0}, the harmonic map heat flow and sequence of harmonic maps. However, in our case, the minimizing sequence $\{u_k\}$ does not have such property. Hence, we need to employ other methods to do the blowup analysis.
Fortunately, the minimizing property in homotopy class and the scaling invariance of the functional $E_m(u)$ are sufficient for us to blowup a non-constant $m$-polyharmonic map from $\R^{2m}$ to $N$.
We also find that doing the local scaling of maps as in the theory of harmonic maps etc., i.e., $\big \{u_k(q_k+r_kx)\big \}$, is not a good choice in our case. Instead, we will consider scaling the whole manifold in the following.
\end{remark}

Now let us recall the scaling for Riemannian manifold $(M,g)$, i.e., $g_{\lambda}=\lambda^2 g$ for $\lambda>0$. For the convenience of the reader, we collect some simple facts about scaling for metric.
\begin{lemma}\label{lem:scaling_property}
Suppose $(M,g)$ is a Riemannian manifold with metric $g$ and denote $g_{\lambda}=\lambda^2 g$ for $\lambda>0$. Then we have
\begin{enumerate}
\item If $B_{r}(p;g)$ is the geodesic open ball of radius $r$ centered at $p$ on $(M,g)$, then $B_{r}(p;g)$ becomes the geodesic open ball of radius $\lambda r$ centered at $p$ on $(M,g_\lambda)$, i.e.,
    \begin{align*}
    B_{r}(p;g)=B_{\lambda r}(p;g_\lambda).
    \end{align*}
\item If the injectivity radius of $(M,g)$ is $\iota_g$, then the injectivity radius of $(M,g_\lambda)$ is $\lambda \iota_g$, i.e.,
    \begin{align*}
    \iota_{g_\lambda}=\lambda\, \iota_g.
    \end{align*}
\item If the geodesic normal coordinate $\Phi_{p,g}: \B_{r} \rightarrow B_{r}(p;g)$ reads
    \begin{align*}
    \Phi_{p,g}^*g=g_{ij}(x;p) \,dx^i\otimes dx^j, \quad \forall x \in \B_{r},
    \end{align*}
    then the geodesic normal coordinate for geodesic ball $B_{r}(p;g)=B_{\lambda r}(p;g_\lambda)  \subset (M,g_\lambda)$ reads
    \begin{align}\label{eqn:coordinate}
    \Phi_{p,g_\lambda}^* g_{\lambda}= g_{ij}(\lambda^{-1}x; p) \,dx^i \otimes dx^j, \quad \forall x \in \B_{\lambda r}.
    \end{align}
    It is clear that $\Phi_{p,g_\lambda}^* g_{\lambda}$ converges to the Euclidean metric on $\R^{2m}$ in $C^\infty_{loc}(\R^{2m})$ as $\lambda \rightarrow \infty$. That is to say, $(M,g_\lambda)$ locally converges to an Euclidean domain as $\lambda$ goes to infinity.
\end{enumerate}
\end{lemma}

For simplicity, let us write
\begin{align}
g_k:=g_{r_k^{-1}}=r_k^{-2}g.
\end{align}
and denote by $\nabla_k$ the Levi-Civita connection with respect to $g_k$.
For any open set $\Omega \subset M$, we denote by $\|\cdot \|_{W^{k,p}(\Omega;g_k)}$ the $W^{k,p}$-norm with respect to metric $g_k$ on $\Omega$.

\begin{lemma}\label{lem:uk_bound}
For any fixed $R_0 >0$, there exist a positive constant $C$ and a positive integer $N_{R_0}$ such that
\begin{align}
\sup_{k\geq N_{R_0}} \big \|u_k \big \|_{W^{m,2}\big(B_{R_0}(q_k;g_k);g_k \big)} \leq C.
\end{align}
\end{lemma}
\begin{proof}
By Lemma \ref{lem:scaling_property}, we know the injective radius of $(M,g_k)$ is $r_k^{-1} \iota_g$, where $\iota_g$ is the injective radius of $(M,g)$. Due to \eqref{lim:rk_qk}, there exists a positive integer $N_{R_0}$ such that $R_0 <\frac{1}{2} r_k^{-1} \iota_g$ for all $k \geq N_{R_0}$. Hence, the geodesic ball $B_{R_0}(q_k;g_k)$ is well-defined on $(M,g_k)$ for all $k \geq N_{R_0}$.

By a direct calculation, for any $f \in C^\infty(M)$, open set $\Omega \subset M$ and $0 \leq l \leq m$, we have
\begin{align}\label{eqn:f:temp1}
\big \|\nabla^l_k f \big \|_{L^2 \big(\Omega;g_k \big)}
= r_k^{l-m} \big \|\nabla^l f \big \|_{L^2 \big(\Omega;g \big)}, \quad \forall \, k \in \N^+,
\end{align}
which immediately implies that for $k\geq N_0$ there hold
\begin{align}
\big \|\nabla^m_k u_k \big \|_{L^2 \big(B_{R_0}(q_k;g_k);g_k \big)}
&= \big \|\nabla^m u_k \big \|_{L^2 \big(B_{R_0}(q_k;g_k);g \big)}
\leq \|\nabla^m u_k \big \|_{L^2 \big(M;g \big)}\,, \label{eqn:uk:temp1} \\
\big \| u_k \big \|_{L^2 \big(B_{R_0}(q_k;g_k);g_k \big)}
&= r_k^{-m} \| u_k \big \|_{L^2 \big(B_{R_0}(q_k;g_k);g \big)}
= r_k^{-m} \| u_k \big \|_{L^2 \big(B_{R_0r_k}(q_k;g);g \big)} \nonumber \\
& \leq C R_0^m \| u_k \big \|_{L^\infty \big(M;g \big)}\,, \label{eqn:uk:temp2}
\end{align}
where $C$ is a positive constant only dependent of $(M,g)$.

By Gagliardo-Nirenberg interpolation inequality, for all
$f \in C^\infty(M)$, $1\leq l \leq m-1$, $q \in M$, $r \in (0, \frac{1}{2}\iota_g)$ we have
\begin{align}\label{eqn:f:temp2}
r^{l-m}\big \|\nabla^l f \big \|_{L^2 \big( B_r(q;g);g \big)}
\leq C \left(
\big \|\nabla^m f \big \|_{L^2 \big( B_r(q;g);g \big)}
+ r^{-m}\big \| f \big \|_{L^2 \big( B_r(q;g);g \big)}
\right)\,,
\end{align}
where $C$ is a positive constant only depending on $m,l$ and $g$.
Combining (\ref{eqn:f:temp1}) and (\ref{eqn:f:temp2}), for all $1\leq l \leq m-1$, $k \geq N_{R_0}$ we have
\begin{align}
& R_0^{l-m} \big \|\nabla^l_k u_k \big \|_{L^2 \big(B_{R_0}(q_k;g_k);g_k \big)} \notag \\
=& (R_0r_k) ^{l-m} \big \|\nabla^l u_k \big \|_{L^2 \big(B_{R_0}(q_k;g_k);g \big)} \notag \\
=& (R_0r_k) ^{l-m} \big \|\nabla^l u_k \big \|_{L^2 \big(B_{R_0 r_k}(q_k;g);g \big)} \notag \\
\leq & \,C \left(
\big \|\nabla^m u_k \big \|_{L^2 \big(B_{R_0 r_k}(q_k;g);g \big)}
+ (R_0 r_k)^{-m}\big \| u_k \big \|_{L^2 \big(B_{R_0 r_k}(q_k;g);g \big)}\right) \notag \\
\leq & \,C \left(
\big \|\nabla^m u_k \big \|_{L^2 \big(M;g \big)}
+ \big \| u_k \big \|_{L^\infty \big(M;g \big)}
\right)\,. \label{eqn:uk:temp3}
\end{align}
Hence, combining (\ref{eqn:uk:temp1}), (\ref{eqn:uk:temp2}) and (\ref{eqn:uk:temp3}) yields
\begin{align*}
\sup_{k \geq N_{R_0}} \big \|u_k \big \|_{W^{m,2}\big(B_{R_0}(q_k;g_k);g_k \big)} \leq C <\infty,
\end{align*}
where $C>0$ is only dependent on $R_0, m, g$ and the boundedness of $\{\|u_k \|_{W^{m,2}(M,g)} \}$. The proof is completed.
\end{proof}

For any $R_0>0$ fixed, let us choose the geodesic normal coordinate of $(B_{R_0}(q_k;g_k), g_k)$, i.e.,
\begin{align*}
\Phi_{q_k, g_k}: \B_{R_0} \rightarrow B_{R_0}(q_k;g_k).
\end{align*}
By Lemma \ref{lem:scaling_property}, we know that $\Phi_{q_k, g_k}^* g_k$ converges to the Euclidean metric on $\B_{R_0}$  in $C^\infty$ as $k \to \infty$, i.e., if we denote by $\g= \sum_{i} dx^i \otimes dx^i$ the Euclidean metric on $\R^{2m}$, then we have that, for all $l \in \N$, there holds
\begin{align}\label{eqn:metric}
\lim_{k\rightarrow \infty} \| \Phi_{q_k, g_k}^* g_k -\g \|_{C^l(\B_{R_0})}=0.
\end{align}
Then we consider the following sequence
\begin{align}
u_k(x):=u_k\circ \Phi_{q_k, g_k}(x) : \B_{R_0} \rightarrow N.
\end{align}
It follows from (\ref{eqn:metric}) that there exist a positive constant $C>1$ and a positive integer $N_1>N_{R_0}$ (see Lemma \ref{lem:uk_bound}) such that for all $k>N_1$  there holds
\begin{align*}
\frac{1}{C} \| u_k \|_{W^{m,2} \big( B_{R_0}(q_k;g_k);g_k \big)}
\leq \| u_k(x) \|_{W^{m,2} \big(\B_{R_0};\g \big)}
\leq C \| u_k \|_{W^{m,2} \big( B_{R_0}(q_k;g_k);g_k \big)}\,.
\end{align*}
By Lemma \ref{lem:uk_bound}, we know
\begin{align}\label{ineqy:uk_bound}
\sup_{k\geq N_1} \| u_k(x) \|_{W^{m,2} \big(\B_{R_0};\g \big)} \leq C <\infty.
\end{align}
By the arbitrariness of $R_0$ and a diagonal process, we have the following lemma.
\begin{lemma}
There exists a subsequence of $\{u_k(x)\}$, still denoted by $\{u_k(x)\}$, and $v(x) \in W^{m,2}_{loc}(\R^{2m},N)$ such that for any $R_0>0$, $u_k(x)$ converges to $v(x)$ strongly in $W^{m-1,2}(\B_{R_0};\g)$ and weakly in $W^{m,2}(\B_{R_0};\g)$, i.e.,
\begin{align*}
u_k(x) &\rightarrow v(x)\quad \text{in}\,\, W^{m-1,2}_{loc}(\R^{2m};\g),\\
u_k(x) &\rightharpoonup v(x)\quad \text{in}\,\, W^{m,2}_{loc}(\R^{2m};\g).
\end{align*}
\end{lemma}

In what follows, we will prove that $v(x)$ is a non-constant $m$-polyharmonic map from $\R^{2m}$ into $N$.
To begin with, it is easy to verify that measures $\mu_k$, $\xi_k$ (see (\ref{eqn:measure1})) on $M$ satisfy
\begin{align}
\mu_k &= \bigg(  \sum_{i=1}^m |\nabla_k^i u_k|^{\frac{2m}{i}}  \bigg) dM_{g_k}, \\
\xi_k &= \, \big|\Delta_k^{\frac{m}{2}} u_k \big|^2 dM_{g_k}.
\end{align}
Now for any large $R_0>0$ fixed, we only consider the measures $\mu_k, \xi_k$ defined on geodesic ball $B_{R_0}(q_k,g_k)$. If we take the geodesic normal coordinate of $(B_{R_0}(q_k,g_k), g_k)$, we can obtain a sequence of Radon measures on Euclidean ball $\B_{R_0}$, i.e.,
\begin{align*}
\mu_k &=\bigg(  \sum_{i=1}^m |\nabla_k^i u_k|^{\frac{2m}{i}} (x)  \bigg) \sqrt{g_k}(x) dx, \\
\xi_k &=\, \big|\Delta_k^{\frac{m}{2}} u_k \big|^2(x) \sqrt{g_k}(x) dx.
\end{align*}
Hence, we can treat $\mu_k$ and $\xi_k$ as the Radon measures on both $M$ and $\B_{R_0}$.

By \eqref{ineqy:uk_bound}, we can assume, by passing to a further subsequence if necessary, $\mu_k$ and $\xi_k$ weakly converge to Radon measures $\widetilde{\mu}$ and $\widetilde{\xi}$ on $\B_{R_0}$ respectively. It follows from (\ref{eqn:metric}) and Fatou's lemma that there exist Radon measures $\widetilde{\nu}$ and $\widetilde{\eta}$ on $\B_{R_0}$ such that
\begin{align}
\widetilde{\mu}=\widetilde{\nu}+ \widetilde{\mu}^* \quad,
\widetilde{\xi}=\widetilde{\eta}+ \widetilde{\xi}^*
\end{align}
where
\begin{align*}
\widetilde{\mu}^* &= \bigg(  \sum_{i=1}^m |D^i v|^{\frac{2m}{i}}(x)  \bigg) dx, \\
\widetilde{\xi}^* &= \, \big|\Delta^{\frac{m}{2}} v(x) \big|^2 dx.
\end{align*}
where $D, \Delta$, $dx$ denote the ordinary derivative, Laplacian and the standard Euclidean measure respectively.

\begin{lemma}\label{lem:eps_regu_2}
There exists a positive constant $\epsilon_0$ such that for any $x_0 \in \B_{\frac{1}{2}R_0}$, if
\begin{align*}
\widetilde{\mu}(\B_2(x_0)) \leq \epsilon_0,
\end{align*}
then $\widetilde{\nu} \equiv 0$ in $\B_1(x_0)$.
\end{lemma}
\begin{proof}
Firstly, let us introduce the temporary notation $W^{m,2}_g(M,N)$ in place of $W^{m,2}(M,N)$ to emphasize the metric $g$. It follows from (\ref{eqn:f:temp1}) that for any given $k \in \N^+$, there exists a positive constant $C_k>1$ such that for any $f\in C^{\infty}(M,N)$, there holds
\begin{align}
\frac{1}{C_k} \|f \|_{W^{m,2}_{g_k}(M,N)}
\leq  \|f \|_{W^{m,2}_g(M,N)}
\leq C_k  \|f \|_{W^{m,2}_{g_k}(M,N)},
\end{align}
which implies that the critical Sobolev space $W^{m,2}_g(M,N)$ is equivalent to $W^{m,2}_{g_k}(M,N)$ by Theorem \ref{thm:density}. Hence, the homotopy classes in both critical Sobolev spaces are the same.

Secondly, it follows from the scaling invariance of the functional $E_m(u)$ (see (\ref{eqn:scal_inv})) that for any $k \in \N^+$, there holds
\begin{align}\label{eqn:same_inf}
\inf \{E_m(u,g) \,|\, u \in  \alpha \cap W^{m,2}_g (M,N)\}
= \inf \{E_m(u,g_k) \,|\, u \in  \alpha \cap W^{m,2}_{g_k} (M,N)\}\,,
\end{align}
where $\alpha$ denotes the homotopy class in $C(M,N)$ and
\begin{align*}
E_m(v,g_k)=\int_M | \Delta_k^{\frac{m}{2}} v |^2 dM_{g_k}
\end{align*}
where $\nabla_k$ and $\Delta_k$ are the Levi-Civita connection and Laplace-Beltrami operator on $(M, g_k)$ respectively.

Thirdly, (\ref{eqn:metric}) ensures that the following three results hold, which is analogous to Lemma \ref{lem:equi_defect}. For any open set $\Omega \subset \B_{R_0}$, there holds
\begin{enumerate}
\item $u_k(x)$ converges to $v(x)$ strongly in $W^{m,2}(\Omega; \g)$ if and only if $\widetilde{\nu} \equiv 0$ in $\Omega$.
\item If $u_k(x)$ converges to $v(x)$ strongly in $W^{m,2}(\Omega; \g)$, then $\widetilde{\eta} \equiv 0$ in $\Omega$.
\item If $\widetilde{\eta} \equiv 0$ in $\Omega$, then for any $V \subset \subset \Omega$,  $u_k(x)$ converges to $v(x)$ strongly in $W^{m,2}(V; \g)$ and hence $\widetilde{\nu} \equiv 0$ in $V$.
\end{enumerate}
With these useful results at hand, the proof of this lemma is similar to that of Lemma \ref{lem:eps_regularity}. The details are left to the reader.
\end{proof}

Since $r_k \rightarrow 0$ and $q_k \rightarrow p^*$, for any given $R_0>0$ we have that $B_{R_0}(q_k;g_k)=B_{R_0r_k}(q_k;g)$\, $\subset B_{\tau_0}(p^*;g)$ for any sufficiently large $k$, where $\tau_0$ is from \eqref{eqn:p^*}.  By (\ref{eqn:rk_qk}), we have that for any geodesic ball $B_{3r_k}(q;g) \subset B_{R_0r_k}(q_k;g)$, there holds
\begin{align*}
\mu_k \big( B_{3r_k}(q;g) \big)
= \mu_k \big( B_{3}(q;g_k) \big)
\leq \frac{\epsilon_0}{8}.
\end{align*}
By (\ref{eqn:metric}), there exists $L>0$ independent of $q$ such that for all $k>L$ there holds
\begin{align*}
\B_{2}(x_q) \subset \Phi_{q_k, g_k}^{-1} \big( B_{3}(q;g_k) \big)
\end{align*}
where $x_q= \Phi_{q_k, g_k}^{-1}(q) \in \B_{R_0}$. Hence, we obtain
\begin{align}
\widetilde{\mu} \big( \B_{2}(x) \big) \leq \frac{\epsilon_0}{8},
\quad \forall \,x \in \B_{\frac{R_0}{2}}.
\end{align}
This ensures that $u_k(x)$ converges strongly to $v(x)$ in $W^{m,2} \big(\B_{\frac{1}{2}R_0};\g \big)$
by Lemma \ref{lem:eps_regu_2}. By the arbitrariness of $R_0$ and a diagonal process, we know
\begin{align}\label{eqn:uk_v_strong}
u_k(x) \rightarrow v(x)
\quad \mbox{in}\,\, W^{m,2}_{loc}(\R^{2m};\g).
\end{align}
On account of (\ref{eqn:rk_qk}), we know
\begin{align}
\widetilde{\mu}^*(\B_3)= \lim_{k \rightarrow \infty} \mu_k \big(\B_3 \big)
=\lim_{k \rightarrow \infty} \mu_k \big( B_{3}(q_k;g_k) \big)
=\lim_{k \rightarrow \infty} \mu_k \big( B_{3r_k}(q_k;g) \big)
=\frac{\epsilon_0}{8}>0\,,
\end{align}
which implies that $v(x)$ is a non-constant map from $\R^{2m}$ to $N$. Moreover, we can prove that $v(x)$ is a smooth $m$-polyharmonic map from $\R^{2m}$ to $N$.

\begin{theorem}
$v(x)$ is a smooth non-constant $m$-polyharmonic map from $\R^{2m}$ to $N$.
\end{theorem}
\begin{proof}
Since the proof is similar to that of Lemma \ref{lem:u*smooth}, we shall only briefly outline the necessary changes.
For any $\chi \in C^\infty_0 (\R^{2m}, \R^K)$, there exists sufficiently large $R_0>0$ such that $\mbox{supp}\, \chi \subset \B_{\frac{1}{2}R_0}$.
Then for sufficiently large $k$, we can give the following well-defined smooth map on $(M,g_k)$:
\begin{align}
\chi_k(q)= \left\{
\begin{aligned}
 &\chi \big( \Phi_{q_k,g_k}^{-1}(q) \big), && \forall \,q \in B_{R_0}(q_k,g_k)=B_{R_0r_k}(q_k,g), \\
 &0, && \mbox{otherwise}.
\end{aligned}
\right.
\end{align}
Note that the support of $\chi_k$ is contained in $B_{\frac{1}{2}R_0}(q_k,g_k)=B_{\frac{1}{2}R_0r_k}(q_k,g)$.
Now we define
\begin{align}
u_k^t(q):= \pi \bigg( u_k(q) + t \chi_k(q) \bigg), \quad t \in (-\delta_0, \delta_0)\,, \, q\in (M,g_k)\,,
\end{align}
where $\delta_0>0$ is sufficiently small and independent of $k$ and $\pi$ is the nearest point projection. It is clear that $u_k^t \in C^{\infty}(M,N)$ is homotopic to $u_k$, and thus $u_k^t \in \alpha$. We compute
\begin{align*}
E_m(u_k^t, g_k)
&= \int_M \big|\Delta_k^{\frac{m}{2}} u_k^t\big|^2 dM_{g_k} \\
&= E_m(u_k, g_k)
+ 2 t \int_M \big \langle \Delta_k^{\frac{m}{2}}u_k, \Delta_k^{\frac{m}{2}} \big( D\pi(u_k) \chi \big) \big\rangle dM_{g_k} + O_k(t^2) \\
&= E_m(u_k, g)
+ 2 t \int_M \big \langle \Delta_k^{\frac{m}{2}}u_k, \Delta_k^{\frac{m}{2}} \big( D\pi(u_k) \chi \big) \big\rangle dM_{g_k} + O_k(t^2)\,.
\end{align*}
By estimating the second derivative of $E_m(u_k^t, g_k)$ with respect to $t$, we have $|O_k(t^2)| \leq C t^2$ where $C>0$ is independent of $k$. It follows from (\ref{eqn:same_inf}) that by a similar argument to the proof of \eqref{eqn:weaksol_temp1}, we get
\begin{align*}
\lim_{k \rightarrow \infty}
\int_M \big \langle \Delta_k^{\frac{m}{2}}u_k, \Delta_k^{\frac{m}{2}} \big( D\pi(u_k) \chi \big) \big\rangle\, dM_{g_k} =0.
\end{align*}
Now (\ref{eqn:metric}) and (\ref{eqn:uk_v_strong}) are sufficient to guarantee that the rest of proof runs through as in the proof of Lemma \ref{lem:u*smooth}, and finally it yields
\begin{align*}
0= & \lim_{k \rightarrow \infty}
\int_M \big \langle \Delta_k^{\frac{m}{2}}u_k, \Delta_k^{\frac{m}{2}} \big( D\pi(u_k) \chi \big) \big\rangle\, dM_{g_k} \\
=& \int_{\B_{R_0}} \bigg \langle \Delta^{\frac{m}{2}}\big(v(x)\big), \Delta^{\frac{m}{2}} \big( D\pi(v(x)) \chi(x) \big) \bigg\rangle\, dx\\
=& \int_{\R^{2m}} \bigg \langle \Delta^{\frac{m}{2}}\big(v(x)\big), \Delta^{\frac{m}{2}} \big( D\pi(v(x)) \chi(x) \big) \bigg\rangle\, dx,
\end{align*}
which implies that $v(x)$ is a weakly $m$-polyharmonic map from $\R^{2m}$ into $N$. Due to the regularity result of Gastel and Scheven \cite{GS}, $v(x)$ is smooth on $\R^{2m}$. The proof is completed.\qedhere
\end{proof}


\end{document}